\documentclass[leqno,11pt]{amsart}
\usepackage[left=1.4in,right=1.4in,bottom=1in,top=1in]{geometry}

\theoremstyle{plain}
\newtheorem{theorem}{Theorem}
\newtheorem{lemma}[theorem]{Lemma}
\newtheorem{corollary}[theorem]{Corollary}
\newtheorem{definition}[theorem]{Definition}
\numberwithin{theorem}{section}
\numberwithin{equation}{section}

\newcommand{\gM}{\mathfrak{M}}

\newcommand{\C}{\mathbb{C}}
\newcommand{\N}{\mathbb{N}}
\newcommand{\R}{\mathbb{R}}
\newcommand{\T}{\mathbb{T}}
\newcommand{\Z}{\mathbb{Z}}

\newcommand{\cA}{\mathcal{A}}
\newcommand{\cH}{\mathcal{H}}
\newcommand{\cK}{\mathcal{K}}

\newcommand{\GL}{\mathrm{GL}}
\newcommand{\SO}{\mathrm{SO}}
\newcommand{\SU}{\mathrm{SU}}
\newcommand{\U}{\mathrm{U}}

\DeclareMathOperator{\supp}{supp}
\newcommand{\AP}{\mathrm{AP}}

\newcommand{\RUC}{\mathrm{RUC}_b}
\newcommand{\LUC}{\mathrm{LUC}_b}
\newcommand{\UC}{\mathrm{UC}_b}

\newcommand{\du}[2]{\langle#1,#2 \rangle}

\begin{document}
\title[Invariant Means]{A survey of amenability theory for direct-limit groups}

\author{Matthew Dawson}
\address{Department of Mathematics, Louisiana State University, Baton Rouge, LA 70803, U.S.A.}
\curraddr{Centro de Investigaci\'{o}n en Matem\'{a}ticas,  Jalisco s/n, Col. Valenciana, Guajauato, GTO 36240, M\'{e}xico} \email{matthew.dawson@cimat.mx}

\author{Gestur \'{O}lafsson}
\address{Department of Mathematics, Louisiana State University, Baton Rouge, LA 70803, U.S.A.}
\email{olafsson@math.lsu.edu}
\thanks{The research of G. \'Olafsson was supported by NSF grant DMS-1101337.  The research of M. Dawson was partially supported by the NSF VIGRE grant at LSU}


\begin{abstract}
We survey results from amenability theory with an emphasis on applications to harmonic analysis on direct-limit groups.
\end{abstract}

\maketitle
 
\section{Introduction}
\noindent The class of all topological groups is far too general to admit many useful general theorems.  For that reason, a few simplifying assumptions are almost always made in the literature when studying them, so that the groups are ``nice'' enough.  The assumptions come in two broad classes:  separability assumptions, which assure that the topology is ``rich enough,'' as well as compactness and countability assumptions, which assure that the group is not ``too big.''  For instance, topological groups are nearly always assumed to be Hausdorff.

However, the most important assumption commonly made, 
 especially for the purpose of harmonic analysis, is local compactness.  Locally compact groups (and \textit{only} locally compact groups) admit Borel measures invariant under the group action.  This, in turn, provides the existence of a group $C^*$-algebra that carries all of the information from the representation theory of the group. 

Most of the literature on representation theory and harmonic analysis is written from the point of view of separable locally compact groups, even in places where it may not be necessary to make that requirement.  Unfortunately, infinite-dimensional Lie groups are never locally compact.  Hence there is no Haar measure and no hope of a Plancherel formula.  The standard tools of harmonic analysis, such as convolutions and group $C^*$-algebras, appear to break down completely.  

While harmonic analysis may appear at first glance to be at a dead end for infinite-dimensional groups, a surprising amount of progress has been made on their representation theory.  For instance, all separable unitary representations of the full unitary group $U(\cH)$ of a separable Hilbert space $\cH$ have been classified (\cite{PickrellUnitary}).  For direct limits of classical Lie groups, there have been several important results (see, for instance,~\cite{Ol1990}).

More recently, there has also been some progress in finding a good context for studying harmonic analysis on direct-limit groups.  For instance, \cite{Grundling} constructed a suitable group $C^*$-algebra for certain direct-limit groups.  See also \cite{LL}.  A very natural construction has also been used by Olshanski, Borodin, Kerov, and Vershik to prove a sort of Plancherel theorem for certain direct-limit groups (\cite{BO, KOV}). The basic ideas of the construction of the regular representation seem to originate from a paper by Pickrell (\cite{PickrellGrassmann}).

We briefly summarize the construction here for the purposes of comparison with the ideas discussed in this paper. One begins with an increasing chain $\{G_k\}_{k\in\N}$ of compact groups and considers the direct-limit group $G\equiv \varinjlim G_k = \cup_{k\in \N} G_k$. Denote the inclusion maps by $i_n:G_n\rightarrow G_{n+1}$.  Next, one attempts to construct projections $p_n:G_{n+1}\rightarrow G_n$ for each $n$ such that $p_n$ is $G_n$-equivariant and $p_n\circ i_n = \mathrm{Id}$.  It may not be possible to construct such a collection of projections that are continuous, but they should at least be measurable.  These projections determine a projective limit space $\overline{G} = \varprojlim G_n$, in which $G$ is a nowhere-dense subset.  

Next, one considers the normalized Haar measure $\mu_n$ on $G_n$ for each $n$.  Because each projection $p_n$ is $G_n$-equivariant and the normalized Haar measure on a compact group is unique, one sees that $p_n^*(\mu_{n+1}) = \mu_n$ for each $n\in\N$.  This in turn produces a $G_n$-equivariant isometric embedding $L^2(G_n)\rightarrow L^2(G_{n+1})$.  Under the right technical conditions, the projective family of probability measures $\{\mu_n\}_{n\in\N}$ induces a limit measure $\mu$ on $\overline{G} = \varinjlim G_k$ by Kolmogorov's theorem in such a way that $L^2(\overline{G}) = \varinjlim L^2(G_n)$, where the latter injective limit is taken in the category of Hilbert spaces.

  Finally, one notes that $G$ acts by translations on $L^2(\overline{G})$, producing a unitary representation which may then be decomposed. This representation is in a natural way a generalization for the regular representation defined in terms of the Haar measure for a locally compact group.  This program has been carried out for the infinite unitary group $\U(\infty) = \cup_{n\in\N} \U(n)$ (see \cite{BO}) and the infinite symmetric group $S_\infty = \cup_{n\in\N} S_n$ (see \cite{KOV}).    
  The disadvantage of this approach, however, is that it gives information about functions on $\overline{G}$, which is a very different space from $G$ (in particular, the set $G$ considered as a subset of $\overline{G}$ has measure $0$).
  
  An alternate path towards harmonic analysis on direct limits of compact groups is provided by the theory of invariant means and amenability.  Amenability theory is actually a very old subject within mathematics.  Unfortunately, many of the most interesting results are true only for locally compact groups.  Nevertheless, we will see that it is possible to define unitary ``regular representations'' for direct limits of compact groups which provide a certain decomposition theory for functions defined on the group $G$ itself (see also~\cite{Bel}).  Amenability theory has also been used to develop a generalization of the construction of induced representations (see \cite{OW,W}).
  
  This paper aims to collect in one place some of the most relevant results from amenability theory as applied to direct-limit groups.    We begin in Section 2 with a review of the basic functional-analytic and topological properties of means. Section 3 discusses some of the basic properties of amenable groups, including fixed-point theorems. In Section 4 we use amenability to explore the question of which Hilbert space representations of a group are unitarizable. In Section 5 we construct a generalization of the regular representation using invariant means and explore some properties of these representations.  Section 6 reviews the theory of almost-periodic functions.   Section 7 shows the resulting ``Plancherel Theorem'' for direct limits of locally compact abelian groups.  Finally, in Section 8 we discuss the application of invariant means to spherical analysis on direct limits of Gelfand pairs.  In particular, we show how several well-known results about such pairs can be motivated by invariant means.

For treatments of the classical theory of invariant means on locally compact groups, we refer the reader to \cite{ATP,JPP}.  Brief overviews of amenability theory may also be found in \cite{BHV,Bel,Fremlin}.  For more functional-analytic approaches to studying functions on a (not-necessarily locally compact) topological group, see \cite{Bel,Grundling, LL}. 



\section{Means on Topological Groups}
\label{meansSection} 

\noindent Consider the space $l^\infty(G)$ of all bounded functions $f:G\rightarrow\C$, with norm $||f|| = \sup_{g\in G} f(g)$ and involution given by $f^*(g) = \overline{f(g)}$ for all $g\in G$.  Then $l^\infty(G)$ is a commutative $C^*$-algebra.  In fact, $l^\infty(G)$ is a representation of $G$ under the $L$ given by $L_g f(h) = f(g^{-1}h)$ for $g\in G$ and $f\in l^\infty(G)$.  One can also consider the action $R$ given by $R_gf(h) = f(hg)$.  

Because we are interested in continuous representations of $G$, it is natural to consider the space $\RUC(G)$ of functions $f\in l^\infty(G)$ such that the map $G\mapsto l^\infty(G)$, $g\mapsto L_g(f)$ is continuous.  These functions are precisely the bounded uniformly continuous functions on $G$ for the right-uniformity (see \cite{Kelley} for more on uniform spaces).  Similarly, one defines $\LUC(G)$ to be the space of functions $f\in l^\infty(G)$ such that $G\mapsto l^\infty(G)$, $g\mapsto R_g(f)$ is continuous.  Finally, we define the space of bi-uniformly continuous functions to be $\UC(G)=\RUC(G)\cap \LUC(G)$.  One has that $L$ provides a continuous representation of $G$ on $\RUC(G)$, that $R$ is a continuous representation of $G$ on $\LUC(G)$. 

Now suppose that $\cA$ is a closed $C^*$-subalgebra of the space $l^\infty(G)$ of all bounded functions on $G$ which contains the constant functions.  A \textbf{mean} on $\cA$ is a continuous linear functional $\mu\in \cA^*$ such that 
\begin{enumerate}
\item $\mu(\mathbf{1}) = 1$
\item $\mu(f)\geq 0$ if $f\in \cA$ and $f\geq 0$.
\end{enumerate}
In other words, $\mu$ is a state for the $C^*$-algebra $\cA$.  It immediately follows that if $f\in \cA$ with $m\leq f(g) \leq M$ for all $g\in G$, then 
\[
m\leq \mu(f) \leq M.
\]


We write $\gM(\cA)$ for the space of all all means on $\cA$.  Note that $\gM(\cA)$ is contained in the closed unit ball of the dual space $\cA^*$ of all continuous linear functionals on $A$.  Furthermore, it is clear that $\gM(\cA)$ is a weak-$*$ closed, convex subset of $B_1(\cA^*)$.  From the Banach-Alauglu Theorem, it follows that $\gM(\cA)$ is weak-$*$ compact, convex subset of $\cA^*$.  We warn the reader, however, that unless $\cA$ is separable, it is not true in general that $\gM(\cA)$ is sequentially compact.  This subtlety has some important consequences, as we will later see.

For each $g\in G$, we may define a mean $\delta_g\in \gM(\cA)$ by $\delta_g(f) = f(g)$ for each $f\in \cA$. We refer to these means as \textbf{point evaluations}.  We will soon see that these point evaluations generate, in a certain sense all means on $\cA$.  We begin with a lemma about means on $l^\infty(G)$.

\begin{lemma}
\label{extremalPoints}
The means $\delta_g\in \gM(l^\infty(G))$ for $g\in G$ are precisely the extremal points of $\gM(l^\infty(G))$.
\end{lemma}
\begin{proof}
Suppose that $\mu\in\gM(l^\infty(G))$ is an extremal point.  Suppose that $A\subset G$, and define $\mu_A(f)=\mu(1_A f)$.  If $\mu_A(\textbf{1})\neq 0$, then we see that $\widetilde{\mu}_A = \frac{1}{\mu_A(\textbf{1})} \mu_A$ is a mean in $\gM(l^\infty(G))$. We define the zero set of $\mu$ to be the set
\[
Z_\mu = \bigcup_{A\subseteq G \text{ s.t.} \mu_A(\textbf{1}) = 0} A.
\]
We can then define the support of $\mu$ to be the set
\[
\supp \mu = G\backslash  Z_\mu.
\]
Note that $\mu_B(\textbf{1})\neq 0$ for all non-empty $B\subseteq \supp \mu$ and that $\mu(f) = \mu(1_{\supp\mu} f)$ for all $f\in \cA$.

It is clear that $\supp \mu \neq \emptyset$; in fact, if $\supp \mu = \emptyset$, then $\mu(\textbf{1}) =0$, which contradicts the assumption that $\mu(\textbf{1})=1$.  

Now suppose that $\supp \mu \neq \emptyset$ contains at least two elements.  Then we can write $\supp \mu = A \mathbin{\dot{\cup}} B$, where $A,B\neq \emptyset$.  It follows that
\begin{equation}
\label{extremalContradiction}
\mu = \mu_{A} + \mu_{B}=  \mu_A(\textbf{1}) \widetilde{\mu}_A + \mu_B(\textbf{1})\widetilde{\mu}_B
\end{equation}
Then $\widetilde{\mu}_A(1) + \widetilde{\mu}_B(1) = \mu(1_{\supp\mu}1) = 1$.  Furthermore, $\mu_A\neq \mu_B$ (for instance, $\mu_A(1_A) = 1$ but $\mu_B(1_A) = 0$).  Thus \ref{extremalContradiction} contradicts the assumption that $\mu$ is an extremal point.  It follows that $\supp\mu = \{g\}$ for some $g\in G$, and hence $\mu = \delta_g$.
\end{proof}

\begin{theorem}
\label{deltaDensity}
The convex hull of the point evaluations are weak-$*$ dense in  $\gM(\cA)$.
\end{theorem}

\begin{proof}
Because $\gM(l^\infty(G))$ is a compact, convex subset of $l^\infty(G)^*$ under the weak-$*$ topology, the result follows for the case $\cA=l^\infty(G)$ by Lemma~\ref{extremalPoints} and the Krein-Milman theorem.

For each $f\in l^\infty(G)$, we see that $f = (f+||f||_\infty\textbf{1}) - ||f||_\infty\textbf{1}$, where $f+||f||_\infty\textbf{1} \geq 0$ and $||f||_\infty\textbf{1}\in \cA$.  Thus, the hypotheses of the M.\ Riesz Extension Theorem are satisfied, and it follows that every positive functional on $\cA$ may be extended to a positive functional on $l^\infty(G)$.  In particular, every mean on $\cA$ may be extended to a mean on $l^\infty(G)$.  Because the convex hull of point evaluations is weak-$*$ dense in $l^\infty(G)$, the theorem follows.
\end{proof}

We warn the reader that despite Theorem~\ref{deltaDensity}, one can not in general say that an arbitrary mean in $\gM(A)$ is a weak-$*$ limit of convex combinations of point evaluations unless $\gM(\cA)$ is first-countable.

Means may be thought of as a generalization of the notion of probability measures on $G$.  In fact, it possible to view means on $\cA$ as legitimate probability measures on a certain compactification of $G$. 

 Denote by $\widehat{A}$ the space of all characters on $G$ under the weak-$*$ topology on $\cA^*$. Then $\widehat{A}$ is a compact Hausdorff space.  Recall that the Gelfand Transform provides an isomorphism $\verb!^!:\cA\rightarrow C(\widehat{A})$, given by $\widehat{f}(\lambda) = \lambda(f)$ for each $f\in\cA$ and $\lambda\in\widehat{A}$.  It follows that there is a linear isomorphism between the states (that is, means)  of $\cA$ and the states of $C(\widehat{A})$.  But the states of $C(\widehat{A})$ are precisely the Radon measures on $\widehat{A}$ by the Riesz Representation Theorem. 
 
  Thus, the means on $\cA$ correspond bijectively to \textit{measures} on $\widehat{\cA}$.  For a given mean $\mu$ on $\cA$, we will denote the corresponding measure on $\widehat{\cA}$ by $\widehat{\mu}$.  Due to the fact that
  \[
  \mu(f) = \int_{\widehat{\cA}} \widehat{f}(x) d\widehat{\mu}(x),
 \]
 we will occasionally use the notation
  \[
  \mu(f) \equiv \int_G f(x) d\mu(x)
  \]
for $\mu\in\gM(\cA)$ and $f\in\cA$. This notation is slightly misleading because $\mu$ is not, in fact, a measure on $G$.   Nevertheless, we note that $\mu$ does share some properties with integrals; namely, it is linear and satisfies the inequality $|\mu(f)| \leq \mu(|f|)\leq ||f||_\infty$.  
 
Each point-evaluation functional $\delta_g: f\mapsto f(g)$ for $g\in G$ defines a mean in $\gM(\cA)$.  We write $i_\cA:G\rightarrow \gM(\cA)$ for the map $g\mapsto \delta_g$ and denote its image by $G_\cA=\{\delta_g| g\in G\}$.   
 
\begin{lemma}
\label{characterDensity}
The set $G_\cA$ is dense in $\widehat{\cA}$.
\end{lemma}
\begin{proof}
We denote the point evaluation $\delta_g$ by $\widehat{g}$ when we wish to consider it as an element of $\widehat{A}$.  Note that, when considered as a mean in $\gM(\cA)$, the point evaluation $\delta_g$ corresponds under the Gelfand transform to the point measure $\delta_{\widehat{g}}$ on $\widehat{A}$; in fact, if $f\in \cA$, then 
$\delta_g(f)=f(g) =  \widehat{f}(\widehat{g})$.  It follows from Theorem~\ref{deltaDensity} that the convex hull of the point measures $\delta_{\widehat{g}}$ is dense in the space of probability measures on $\widehat{A}$.  In particular, every point measure $\delta_x$ on $\widehat{A}$, where $x\in\widehat{A}$, is in the closed convex hull of the point measures $\delta_{\widehat{g}}$.  

Suppose that $x\in\widehat{A}$ is not in the closure of $G_\cA$.  Then Urysohn's Lemma implies the existence of a continuous function $h\in C(\widehat{A})$ such that $h=0$ on the restriction to the closure of $G_\cA$ but $h(x)=1$.  Hence $\delta_x(h) = 1$ but $\mu(h) = 0$ for every measure $\mu$ in the closed convex hull of the point measures $\delta_{\widehat{g}}$, which contradicts the weak-$*$ density noted above.    
\end{proof}

 It is clear that $i_\cA$ is injective if and only if $\cA$ separates points on $G$.  We will later see some examples for which $i_\cA$ is far from being injective.
 
 We end this section with three topological results about the compactification $i_\cA:G\rightarrow \widehat{\cA}$.

\begin{lemma}
If $\cA\subseteq C(G)$, then $i_\cA$ is continuous. 
\end{lemma}
\begin{proof}
A basis for the weak-$*$ topology on $\widehat{\cA}$ is given by the  neighborhoods 
\[
B_{f_1,\ldots f_k}^\epsilon(\delta_h) = \left\{\left.x\in \widehat{\cA} \right| |x(f_i)-\delta_h(f_i))|<\epsilon \text{ for all } 1\leq i\leq k \right\},
\]
where $f_i,\ldots f_k \in \cA$, $h\in G$, and $\epsilon > 0$, provide a basis for the weak-$*$ topology on $\cA$ (here we use that $G_{\cA}$ is dense in $\widehat{\cA}$).  Pick $g\in i_\cA^{-1}(B_{f_1,\ldots f_k}^\epsilon(\delta_h))$.  Then $|f_i(g)-f_i(h)|<\epsilon$ for $1\leq i\leq k$.  Since $f_1,\ldots f_k\in C(G)$, it follows that there is an open neighborhood $V$ of $g$ such that $|f_i(a)-f_i(h)|<\epsilon$ for all $a\in V$.  It is clear that $V\subseteq i_\cA^{-1}(B_{f_1,\ldots f_k}^\epsilon(h))$ and we are done.
\end{proof}

\begin{theorem}
Suppose that $\cA$ separates closed subsets of $G$ (for instance, if $\cA=l^\infty(G)$ or if $G$ has a normal topology and $C(G)\subseteq \cA$).  If $U$ and $V$ are disjoint closed subsets of $G$, then $i_\cA(U)$ and $i_\cA(V)$ have disjoint closures in~$\widehat{\cA}$.
\end{theorem}
\begin{proof}
Suppose that $U$ and $V$ are disjoint closed subsets of $G$ and that $f:G\rightarrow [0,1]$ is an element of $\cA$ such that $f|_U=0$ and $f|_V = 1$.  Thus, for any $g\in U$ and $h\in V$, one has that $\delta_g(f) = 0$ and $\delta_h(f)=1$.  It immediately follows that $x(f) = 0$ for each $x\in \overline{i_\cA(U)}$ and that $y(f) = 1$ for any $y \in \overline{i_\cA(V)}$.  The result then follows.  
\end{proof}

\begin{theorem}
\label{homeomorphismTheorem}
Suppose that for each closed set $F\subseteq G$ and each point $g\in G\backslash F$ there is a function $f\in\cA$ such that $f|_F=0$ but $f(g)=1$.  Then $i_\cA$ is a homeomorphism onto its image.
\end{theorem}
\begin{proof}
Pick a closed set $F\subseteq G$, a point $g\notin F$ and a function $f\in\cA$ such that $f|_F=0$ but $f(g)=1$. Then $\delta_g(f) = 1$ and $x(f) = 0$ for all $x$ in the closure $\overline{i_\cA(F)}$.  Then $\delta_g\notin  \overline{i_\cA(F)}$ for all $g\in G\backslash F$.  In particular, $i_\cA(F)$ is closed in the relative topology of $G_\cA$.  Thus, $i_\cA$ is a closed map onto its image.  Since $\cA$ separates points on $G$, we have that $i_\cA$ is a continuous injection.  The result follows.
\end{proof}

\begin{corollary}
The map $i_{\RUC(G)}: G\rightarrow \widehat{\RUC(G)}$ is a homeomorphism onto its image.
\end{corollary}
\begin{proof}
Because $G$ is a topological group, it is also a uniform space using the left action of the group on itself.  Its topology is thus given by a family of semimetrics.  It follows that there are sufficient right-uniformly continuous functions on $G$ to separate closed sets from points.  The result then follows from Theorem~\ref{homeomorphismTheorem}.
\end{proof}

\section{Amenable Groups}
\noindent For the rest of this article, we assume that  $\cA\subseteq \RUC(G)$.  One can then define a continuous left-action of $G$ on $\gM(\cA)$ by setting $g\cdot \mu(f) = \mu(L_{g^{-1}} f)$. We say that a mean $\mu\in\gM(\cA)$ is an \textbf{invariant mean on $\cA$} if $g\cdot \mu = \mu$ for all $g\in G$ (that is, if $\mu(L_g f) = \mu(f)$ for all $g\in G$).  One says that $G$ is an \textbf{amenable group} if there is a nontrivial invariant mean in $\gM(\RUC(G))$.

It is not difficult to show that every character in  $\widehat{\cA}$ is a positive functional on $\cA$ and that, in fact, $\widehat{\cA}$ is a closed subset of $\gM(\cA)$ in the weak-$*$ topology. Furthermore, the $G$-action on $\gM(\cA)$ restricts to a continuous left-action of $G$ on $\widehat{\cA}$, and so it is possible to ask whether there are any $G$-invariant characters in $\widehat{\cA}$.  
Recalling from Section~\ref{meansSection} the correspondence between Radon measures on $\widehat{\cA}$ and means on $\cA$, we see that there is a one-to-one correspondence between $G$-invariant means in $\gM(\cA)$ and $G$-invariant measures on $\widehat{\cA}$.


If there is a nontrivial invariant mean in $\gM(\RUC(G))$, then we say that $G$ is \textbf{amenable}.  The term ``amenable'' is due to M. M. Day, who discovered a very powerful alternate characterization of amenablility in terms of the affine actions of a group on compact convex sets.  (An \textbf{affine action} of a group $G$ on a convex subset of $K$ of a vector space $V$ is an action $G\times K\rightarrow K$, $(g,v)\mapsto g\cdot v$ such that $g\cdot (tv + (1-t) w) = t(g\cdot v) + (1-t)(g\cdot w)$ for all $v,w\in K$, $g\in G$, and $t\in [0,1]$.)

\begin{theorem}[Day's Fixed Point Theorem \cite{Day2}]
Let $G$ be a topological group.  The following are equivalent:
\begin{enumerate}
\item $G$ is amenable.
\item Every compact Hausdorff space on which $G$ acts continuously admits a $G$-invariant Radon probability measure.
\item Every continuous affine action of $G$ on a compact convex subset  $K$ of a locally convex vector $V$ space has a fixed point.
\end{enumerate}
\end{theorem}
\begin{proof}
We closely follow the proof in \cite[Theorem G.1.7]{BHV}.
We begin with $(1)\implies (2)$.  Let $K$ be a compact Hausdorff space with a continuous $G$-action, and let $\mu\in \gM(\RUC(G))$ be a $G$-invariant mean.  Fix a point $v\in K$ and consider the continuous map $p:G\rightarrow K$ by $p(g)= g\cdot v$. We define a measure $\mu_K$ on $K$ by setting
\[
\mu_K(\phi) = \mu(\phi\circ p)
\]
for each $\phi\in C(X)$.  One shows that $\phi\circ p\in \RUC(G)$.  It is clear that $\mu_K$ defines a continuous positive functional on $C(X)$ with total mass $||\mu_K||=\mu_K(\textbf{1})=1$.  Thus $\mu_K$ defines a $G$-invariant probability measure on $K$ by the Riesz representation theorem.  

To prove $(2)\implies (3)$, let $K$ be a compact convex subset of a locally convex vector space $V$, and suppose that $G$ acts affinely on $K$.  In particular, by $(2)$ there is a $G$-invariant probability measure $\mu_K$.

Now let $\mu$ be a Radon measure on $K$.  By \cite{Rudin}, there is $b_\mu\in K$ such that 
\[
\langle b_\mu,\lambda \rangle = \int_K \langle v,\lambda \rangle d\mu_K(v)
\]
for each $\lambda\in V^*$.  One refers to $b_\mu$ as the \textbf{barycenter} of $\mu$.

Note that the space of all Radon measures on $K$ is identical to the space of all means on $C(K)$.  Suppose that $\mu = c_1 \delta_{v_1}+ \cdots +c_k\delta_{v_k}$, where $c_1,\ldots c_k\geq 0$ with $c_1+\cdots +c_k = 1$, and where $v_1,\ldots, v_k\in K$.  $ \sum_{i=1}^k c_i v_i$.  If $g\in G$, then $g\cdot \mu = c_1 \delta_{g\cdot v_1} + \cdots c_k\delta_{g\cdot v_k}$.  Thus $b_{g\cdot \mu} = g\cdot b_\mu$.  Since the convex hull of point measures on $K$ is weak-$*$ dense in the space of all measures on $K$, it follows that $b_{g\cdot\mu} = g\cdot b_\mu$ for all $g\in G$ and all Radon measures $\mu$ on $K$.

In particular, for the $G$-invariant measure $\mu_K$ on $K$, we see that $g\cdot b_{\mu_K} = b_{g\cdot \mu_K} = b_{\mu_K}$.  Hence $b_{\mu_K}$ is a $G$-fixed point in $K$.

Finally, to see that $(3)\implies (1)$, we need only note that $G$ acts continuously on the compact convex subset $\gM(\RUC(G))$ of the vector space $\RUC(G)^*$.  Thus, $(3)$ immediately implies the existence of a $G$-invariant mean in $\gM(\RUC(G)$.
\end{proof}

We immediately arrive at the following corollary:
\begin{corollary}
Suppose that $G$ is an amenable group.  Then there is a right-$G$-invariant mean on $\LUC(G)$.  Furthermore, there is a mean in $\gM(\UC(G))$ which is both right- and left-$G$-invariant and invariant under the transformation $f\mapsto f^\vee$ given by $f^\vee(x) = f(x^{-1})$.
\end{corollary}

The following theorem collects many of the most important lemmas for constructing amenable groups. 

\begin{theorem}(\cite[Proposition G.2.2]{BHV}, \cite[Theorem 449C, Corollary 449F]{Fremlin})
Let $G$ be a topological group.  Then
\begin{enumerate}
\item If $G$ is compact, then $G$ is amenable.
\item If $G$ abelian, then $G$ is amenable.
\item If $G$ is amenable and $H$ is a closed normal subgroup of $G$, then $G/H$ is amenable.
\item If $G$ is amenable and $H$ is an open subgroup of $G$, then $H$ is amenable.
\item If $G$ is amenable and $H$ is a dense subgroup of $G$, then $H$ is amenable.
\item If $H$ is a closed normal subgroup of $G$ such that $H$ and $G/H$ are amenable, then $G$ is amenable.
\end{enumerate}
\end{theorem}
\begin{proof}
To prove (1), we note that if $G$ is compact, then $\RUC(G)=C(G)$.  Thus Haar measure provides an invariant mean on $\RUC(G)$.

Statement (2) is the Markov-Kakutani fixed-point theorem. See \cite[Theorem G.2.1]{BHV} for a proof.

To prove (3), we suppose that $G$ is amenable and $H$ is a closed normal subgroup.  Write $p:G\rightarrow G/H$ for the canonical quotient map.  We claim that if $f\in \RUC(G/H)$, then $f\circ p\in \RUC(G)$.  In fact, if $||L_{gH} f-f||_\infty <\epsilon$ for all $gH$ in some neighborhood $V$ of $eH$, then $||L_g (f\circ p)-f\circ p||_\infty <\epsilon$ for all $g\in p^{-1}(V)$, and the claim follows.  Furthermore, $||f\circ p||_\infty = ||f||_\infty$.  Now let $\mu_G$ be an invariant mean on $\RUC(G)$.  Then we may define a mean $\mu_{G/H}$  on $\RUC(G/H)$ by setting $\mu_{G/H}(f) = \mu_G(f\circ p)$ for all $f\in \RUC(G)$.  It is clear that $\mu_{G/H}$ is $G/H$-invariant.  Thus $G/H$ is amenable.  

See \cite[Corollary 449F]{Fremlin} for the proofs of $(4)$ and $(5)$.

To prove (6), we suppose that $H$ is a closed normal subgroup of $G$ such that $H$ and $G/H$ are amenable.  Let $G$ act continuously and affinely on a compact convex subset $K$ of a locally-convex vector space $V$.   We denote by $K^H$ the set of all $H$-fixed points in $K$.  Because $H$ is amenable, Day's theorem shows that $K^H$ is nonempty.  It is not difficult to show that $K^H$ is a closed, convex subset of $V$.  Furthermore, the action of $G$ on $K^H$ factors through to a well-defined, continuous action of $G/H$ on $K^H$ defined by $gH\cdot v = g \cdot v$ for all $v\in K$ and $gH\in G/H$.  Finally, $K^H$ must possess a $G/H$-fixed point $x$ because $G/H$ by Day's Theorem because $G/H$ is amenable.  Thus $x$ is a $G$-fixed in $K$.  Hence $G$ is amenable.

\end{proof}

It is well-known that every closed subgroup of an amenable \textit{locally compact} group is amenable.  We caution the reader, however, that this result is not true for groups which are not locally-compact (see \cite[p. 457] {BHV}).

There is one more well-known method of constructing amenable groups which is of particular interest to us here: any direct limit of amenable groups is again amenable.

\begin{theorem}(\cite[Proposition 13.6]{JPP}).
\label{directLimitsAmenable}
Suppose that $I$ is a linearly-ordered index set and that $\{G_n\}_{n\in I}$ is an increasing chain of amenable subgroups (that is, $G_n\leq G_m$ if $n\leq m$).  Then $G_\infty \equiv \varinjlim G_n = \cup_{n\in I} G_n$ is amenable. 
\end{theorem}

\begin{proof}
For each $n\in I$, choose an invariant mean $m_n$ for $G_n$.  We then define a functional $\mu_n\in\mathrm{\RUC}(G_\infty)^*$ by
\[
   \mu_n(f) = m_n(f|_{G_n})
\]
for each $f\in\mathrm{\RUC}(G_\infty)$.  Because $f|_{G_n}\in \RUC(G_n)$ for all $f\in \RUC(G_\infty)$, it is clear that each $\mu_n$ is a mean on $G_\infty$.  Furthermore, we see that $\mu_n(L_g f) = \mu_n(f)$ whenever $g\in G_n\leq G_\infty$.  Thus, any weak-$*$ cluster point of the set $\{\mu_n\}_{n\in I}\subseteq \mathrm{\RUC}(G_\infty)^*$ would be invariant under $G_\infty = \cup_{n\in I} G_n$.  Furthermore, by the Banach Alaoglu theorem, the unit ball in $\mathrm{\RUC}(G_\infty)^*$ is weak-$*$ compact and thus our sequence must possess a cluster point.
\end{proof}

Because $\mathrm{\RUC}(G_\infty)$ is not separable when $G_\infty$ is not compact, the unit ball in $\mathrm{\RUC}(G_\infty)^*$ is not guaranteed to be weak-$*$ \textit{sequentially} compact.  Thus there is no reason to expect that $\{\mu_n\}_{n\in\N}\subseteq \mathrm{\RUC}(G_\infty)^*$ will possess a convergent sequence.  In fact, an application of the Axiom of Choice is required to construct an invariant mean on $G_\infty$.  

An immediate corollary of Theorem~\ref{directLimitsAmenable} is that every group formed as a direct limit of compact groups is amenable.

\section{Unitarizability}
\label{unitarySection}
\noindent In this section we look at some applications of invariant means to the theory of unitary representations for topological groups.  For a given topological group, it is often very difficult to determine which representations of the group on a Hilbert space are in fact equivalent to unitary representations.  For amenable groups, however, there is a very succinct solution to this question:

\begin{theorem}(\cite[Proposition 17.5]{JPP}).
Suppose that $G$ is an amenable group and that $\pi$ is a  continuous representation of $G$ on a separable Hilbert space $\cH$.  Then $\pi$ is equivalent to a unitary representation if and only if it is uniformly bounded (that is, $ \sup_{g\in U_\infty} ||\pi(g)|| <\infty$).
\end{theorem}

\begin{proof}
 Suppose that $\pi$ is equivalent to a unitary representation. Then there is an invertible bounded intertwining operator $T\in\GL(\cH)$  unitarizes $\pi$.  It follows that $T\pi(g)T^{-1}$ is unitary and thus $||\pi(g)||\leq ||T||||T^{-1}||$ for all $g\in U_\infty$, and thus $\pi$ is uniformly bounded.

To prove the converse, let $M = \sup_{g\in U_\infty} ||\pi(g)||$.  Note that one also has that $M = \sup_{g\in\U_\infty} ||\pi(g)^{-1}||$. It follows that 
\[
    M^{-1} ||u|| \leq ||\pi(g)u|| \leq M ||u||
\]
for all $g\in U_\infty$.

Now let $\mu$ be a bi-invariant mean on $G$.  We denote the inner product on $\cH$ by $\du{\cdot\,}{\cdot}_\cH$. Note that $g\mapsto \langle\pi(g)u,\pi(g)v\rangle_\cH$ is a uniformly continuous, bounded function on $G$ (since $\pi$ is strongly continuous and uniformly bounded).    We may thus define a new inner product $\du{\cdot\,}{\cdot}_\mu$ on $\cH$ by 
\[\langle u,v\rangle_\mu = \int_G \langle \pi(g)u,\pi(g)v \rangle_\cH d\mu(g)\]
for all $u,v\in\cH$, where for clarity we have used the ``integral'' notation for means that was introduced in Section~\ref{meansSection}.  It is clear that $\du{\cdot\,}{\cdot}_\mu$ provides a positive semi-definite Hermitian form on $\cH$.  

Note that for $u\in\cH\backslash\{0\}$ one has that 
\[
   0 < M^{-2}||u||_\cH^2 \leq ||u||_\mu^2 = \int_G ||\pi(g)u||_\cH^2 d\mu(g) \leq M^2||u||_\cH^2.
\]
Thus $\du{\cdot\,}{\cdot}_\mu$ is strictly positive-definite and continuous with respect to $\du{\cdot\,}{\cdot}_\cH$.
\end{proof}

For a compact group $G$, every continuous representation is uniformly bounded, and hence the above theorem amounts to the fact that \textit{every} continuous representation of $G$ on a Hilbert space is equivalent to a unitary representation.  Unfortunately, the same does not hold true for direct limits of compact groups, as we now briefly demonstrate.

Consider the group $U_\infty = \SU(\infty) = \varinjlim \SU(2n)$.  For each $n\in\N$, consider the standard representation $\pi_n$ of $\SU(2n)$ on $\cH_n = \C^{2n}$ (that is, $\pi_n(g)v = g\cdot v$ for all $g\in\SU(2n)$).  By taking the direct limit, we may form a unitary representation $\pi = \varinjlim \pi_n$ of $\SU(\infty)$ on the Hilbert space $\cH = \ell^2(\C) = \overline{\varinjlim \C^{2n}}$ of square-summable sequences of complex numbers.  Note that $\SU(2n)$ acts trivially on the orthogonal complement of $\cH_n$. It follows that $\pi|_{\SU(2n)}$ decomposes into a direct sum of the standard representation $\pi_n$ and infinitely many copies of the trivial irreducible representation.  That is,
\[ 
   \pi|_{\SU(2n)} = \pi_n \oplus \infty\cdot 1_{\SU(2n)},
\]
where $1_{\SU(2n)}$ denotes the trivial irreducible representation of $\SU(2n)$ on $\C$.  


Now let $V_1 = \cH_1$ and define $V_n = \cH_n\ominus \cH_{n-1}$ for each $n>1$. 
Note that $\dim V_n = 2$ for each $n\in\N$.
We now completely discard unitarity and choose some new inner product $\langle,\rangle_{V_n}$ on $V_n$ under which $||\pi(g)|_{V_n}||\geq n$ for some $g\in \SU(2n)$.  For instance, if $\pi(g)v = w$, where $v,w\in V_n$ are linearly independent, then we can choose any inner product $\langle,\rangle_{V_n}$  on $V_n$ such that $||v||_{V_n}=1$ and $||w||_{V_n} = n$.

Next we define for each $n\in\N$ the finite-dimensional Hilbert space
\[
   \cK_n = \bigoplus_{i=1}^n V_i,
\]
where each $V_i$ is given the new inner product we just defined.  As vector spaces, $\cK_n = \cH_n$, but they possess different inner products. Now $\{(\pi_n,\cK_n)\}_{n\in\N}$ forms a direct system of continuous Hilbert representations.  We consider the representation $(\widetilde{\pi}_\infty,\cK_\infty) = (\varinjlim \pi_n,\overline{\varinjlim \cK_n})$. Note that  $\pi|_{\SU(2n)}$ and $\widetilde{\pi}|_{\SU(2n)}$ possess the same irreducible subrepresentations for each $n\in\N$.  
Finally, it is clear that $\widetilde{\pi}$ is not uniformly bounded (since $\sup_{g\in \SU(2n)} ||\pi(g)||\geq n$ for each $n\in\N$), and is therefore not unitarizable.

\section{Invariant Means and Regular Representations}
\noindent Unitary representations for locally compact groups are, of course, closely related to harmonic analysis.  For instance, if $G$ is a locally compact group, then decomposing the unitary regular representation of $G$ on $L^2(G)$ is one of the foundational problems in harmonic analysis.  While groups which are not locally compact do not possess Haar measures, one can develop an $L^2$-theory using invariant means (these definitions may also be found in, for instance, \cite{Bel}).  

In particular, suppose that $G$ is a topological group and $\cA$ is a closed $C^*$-subalgebra of $\RUC(G)$.  For each invariant mean $\mu\in \gM(\cA)$, one can construct a Hilbert space $L^2_\mu(G)$ as follows.  Define a pre-Hilbert seminorm on $\cA$ by \[\langle f,g \rangle_\mu = \mu(f\bar{g}),\]
and set $\cK_\mu=\{f\in \cA | \langle f,f\rangle_\mu = 0\}$.  Then $L^2_\mu(G)$ is defined to be the Hilbert-space completion of $\cA/\cK_\mu$.  

\begin{lemma}
$\cK_\mu$ is a closed subspace of $\cA$.
\end{lemma}
\begin{proof}
Suppose that $\{f_k\}_{n\in\N}$ is a sequence in $\cK_\mu$ which converges to $f$ in $\cA$.  Fix $\epsilon > 0$, and choose $N$ such that $||f_k - f||_\infty< \epsilon$ for $n\geq N$.  Then 
\[
\begin{array}{ll}
\langle f, f \rangle_\mu & = \langle f_k + (f_k - f), f_k + (f_k - f)\rangle_\mu \\
              & = \langle f_k, f_k \rangle_\mu + \langle f_k - f, f_k - f \rangle_\mu + 2 \Re \langle f_k, f_k - f \rangle_\mu \\
              & \leq 0 + \epsilon^2 + 2(||f||_\infty+\epsilon)\epsilon, \\
\end{array}
\]
where we use the fact that $\langle g,h \rangle_\mu = \mu(g\bar{h}) \leq ||g||_\infty ||h||_\infty$ for all $g,h\in \cA$.  Because $\epsilon > 0$ was arbitrary, it follows that $f\in \cK_\mu$.
\end{proof}

Because $\cA\subseteq \RUC(G)$, one sees that the left-regular representation of $G$ on $\cA$ defined by $L_gf(h) = f(g^{-1}h)$ is continuous.  

\begin{theorem}
 The regular representation $L$ of $G$ on $\cA$ descends to a continuous unitary representation on $L^2_\mu(G)$.  
\end{theorem}
\begin{proof}
Because $\mu$ is an invariant mean, it follows that $\cK_\mu$ is a closed invariant subspace of $\cA$.  Hence, $L$ descends to a continuous representation of $G_\infty$ on $UCB(G_\infty)/\cK_\mu$.  

Since $\langle f, f \rangle_\mu \leq ||f||_\infty^2$ for all $f\in \cA$, we see that $L$ is a continuous representation in the pre-Hilbert space topology on $\cA$ and descends to a continuous representation in the pre-Hilbert space topology on $\cA/\cK_\mu$.  Furthermore, $L$ acts by isometries in those pre-Hilbert space topologies due to the fact that $\mu$ is an invariant mean on $G$.  Thus $L$ extends to a continuous unitary representation on the Hilbert-space completion $L^2_\mu(G)$ of $\cA/\cK_\mu$.
\end{proof}

One nice aspect of this approach is that it allows the consideration of an $L^2$-theory of harmonic analysis that is intrinsic to the group, in the sense that it actually provides a decomposition of functions on the group $G$ and not on a larger $G$-space, such as with the projective-limit construction mentioned earlier.  The disadvantage is that it depends heavily on the choice of invariant mean $\mu$ and $C^*$-algebra $\cA$, as we shall see.  Furthermore, the fact that the axiom of choice is necessary in many cases to construct invariant means implies that it may not be possible to construct an explicit decomposition of $L^2_\mu(G)$.

It is also possible that $L^2_\mu(G)$ gives the trivial representation of $G$, in which case very little information may be obtained about the functions on $G$.  In general, as the next theorem demonstrates, one can gain information on the size of $L^2_\mu(G)$ (and therefore gauge how much information may be gleaned about functions on $G$) by determining the support of the corresponding measure $\widehat{\mu}$ on $\widehat{\cA}$.

\begin{theorem}
For a $G$-invariant mean $\mu$ on $\cA$, on has the equivalance of unitary representations
\[
L^2_\mu(G)\cong_G L^2(\widehat{A},\widehat{\mu}),
\]
where $L^2(\widehat{A},\widehat{\mu})$ denotes the unitary representation of $G$ corresponding to the action of $G$ on $\widehat{A}$ under the $G$-invariant measure $\widehat{\mu}$ on $\widehat{A}$.
\end{theorem}
\begin{proof}
The map $\cA\rightarrow C(\widehat{\cA})$, $f\mapsto \widehat{f}$ is clearly a $G$-intertwining operator.  The fact that it extends to the required unitary intertwining operator follows from the fact that $\mu(f) = \int_{\widehat{\cA}} \widehat{f}(x)d\widehat{\mu}(x)$ for all $f\in \cA$.
\end{proof}

\begin{corollary}(\cite[Remark 3.11]{Bel})
One has $\dim L^2_\mu(G) = 1$ if and only if there is a $G$-invariant character $x\in\widehat{\cA}$ such that $\widehat{\mu} = \delta_x$.  In that case, one has the decomposition
\[
\cA = \C \text{\textbf{1}} \oplus \cK_\mu,
\]
where $\C\textbf{1}$ denotes the space of constant functions on $G$.
\end{corollary}

A group $G$ is said to be \textbf{extremely amenable} if there is a $G$-invariant character in $\widehat{\RUC(G)}$.  While it is known that no locally-compact groups are extremely amenable, it has been shown (see \cite{GP,GM}) that $\SO(\infty) = \cup_{n\in\N} \SO(n)$, $\SU(\infty) = \cup_{n\in\N} \SU(n)$, and the infinite symmetric group $S_\infty = \cup_{n\in\N} S_n$ are extremely amenable. One can also show that direct products of extremely amenable groups are again extremely amenable (see \cite[Theorem 449C]{Fremlin}).

 It is evident that if $G$ is an amenable group, then by Day's theorem, the closure of each $G$-orbit in $\widehat{\cA}$ gives rise to an invariant mean in $\widehat{\cA}$. Unfortunately, it is very difficult to study such orbits because of the unavoidable use of the axiom of choice in the construction of $\widehat{\cA}$. In the next section we look at a special case in which one may determine all of the invariant means on $\cA$ and say something about the decomposition of the representation $L^2_\mu(G)$.

\section{Almost Periodic Functions}

\noindent In general, there is no uniqueness property for invariant means similar to the uniqueness of Haar measures.  In fact, for many groups $G$ it has been shown that the set of all invariant means on $\RUC(G)$ has cardinality $2^{2^{|G|}}$ (see \cite{JPP}). However, one might hope for the existence of subalgebras $\cA$ of $\RUC(G)$ which possess a unique invariant mean $\mu \in \gM(\cA)$.  In particular, the algebra of  weakly almost periodic functions on a group always satisfies this property.  In some cases it is even possible to explicitly determine the value of this mean and write down an explicit decomposition of the unitary regular representation $L^2_\mu(G)$.

\begin{definition}
A continuous function $f\in C(G)$ is said to be \textbf{almost periodic} if the set $\{L_g f\}_{g\in G}$ is relatively compact in the norm topology of $C(G)$.  
We denote by $\AP(G)$ the space of almost periodic  functions on $G$.
\end{definition}

Von Neumann proved the following result about almost periodic functions:
\begin{theorem}(Von Neumann \cite{vN})
\label{vonNeumannTheorem}
If $f\in \AP(G)$, then the closed convex hull $\overline{\mathrm{co}}(\{L_g f\}_{g\in G})$ of the set of $G$-translates contains exactly one constant function.  We denote value of this constant function by $M(f)$.  Furthermore, $M$ is a $G$-invariant mean on $\AP(G)$.
\end{theorem}
By continuity arguments, one sees that any invariant mean on $\AP(G)$ must take the same value for every function in $\overline{\mathrm{co}}(\{L_g f\}_{g\in G})$.  Because this set contains a unique constant function, one sees that for any invariant mean $\mu\in\gM(f)$, one has $\mu(f) = M(f)$.  In fact, it is also  not difficult to see that every invariant mean $\lambda$ on $\RUC(G)$ must have the property that $\lambda|_{\AP(G)} = M$.

In fact, it is possible to put a topological group structure on $\widehat{\AP(G)}$ so that $i_{\AP(G)}:G\rightarrow \widehat{\AP(G)}$ is a continuous homomorphism: 

\begin{theorem}(\cite[p. 166, 168]{Loomis})
$\AP(G)$ is a $G$-invariant closed $C^*$-subalge\-bra of $\RUC(G)$. Furthermore, the product defined by $\delta_{g}\cdot \delta_{h} = \delta_{gh}$ on $i_\AP (G)$ extends to a compact topological group structure on $\widehat{\AP(G)}$. 
\end{theorem}

 For the sake of clarity, we introduce the notation $G_c \equiv \widehat{\AP(G)}$. 
Note that the Gelfand isomorphism $\verb!^!: \AP(G)\rightarrow C(G_c)$ sets up a correspondence between almost periodic functions on $G$ and continuous functions on $G_c$.  Furthermore, one sees that the invariant mean on $\AP(G)$ corresponds to the Haar measure on $G_c$.  That is, $M(f) = \int_{G_c} \widehat{f}(x) dx$ for all $f\in \AP(G)$.

 In the notation of Section~\ref{unitarySection}, we see that $\cK_M =0$ because $\int_{G_c}\widehat{f}(x) dx>0$ whenever $f\in \AP(G)$ and $f>0$. In particular, the Gelfand transform extends to a unitary $G$-intertwining operator $\verb!^!: L^2_M(G) \rightarrow L^2(G_c)$, where $L^2(G_c)$ is defined using the Haar measure on $G_c$.  Finally, because $G_{\AP(G)}$ is a dense subgroup of $G_c=\widehat{\AP(G)}$, it follows that a subspace of $L^2(G_c)$ is $G$-invariant if and only if it is $G_c$-invariant. 
 
 The following result now follows immediately from the Peter-Weyl theorem:  
 \begin{theorem}
The representation  $L^2(G_c)$ decomposes into a direct sum of finite-dimensional representations of $G$. 
\end{theorem}

\begin{corollary}
If $G$ possesses no nontrivial finite-dimensional unitary representations, then $\AP(G)$ contains only constant functions on $G$.
\end{corollary}

For instance, it follows that $SL(n,\R)$ and $\SU(\infty) = \cup_{n\in\N} \SU(n)$ have no nontrivial almost-periodic functions, so the decomposition theory for almost periodic functions provides no information for such groups.

The situation is much more interesting if $G$ is abelian.  Denote by $\widehat{G}$ the group of all continuous characters of $G$. Note that any character $\chi\in\widehat{G}$ is almost periodic; in fact, $L_g\chi = \chi(g^{-1})\chi$, so that $\{L_g \chi\}_g\in G$ is a compact subset of a one-dimensional vector space.  Thus, every character $\chi\in \widehat{G}$ corresponds to a continuous function $\widehat{\chi}$ on $G_c = \widehat{\AP(G)}$.  Because $\widehat{\chi}$ is a character on the dense subgroup $G_{\AP(G)}$, it follows that $\widehat{\chi}$ is in fact a continuous character of $G_c$.  In other words, we have shown that:
\begin{theorem}
The Gelfand transform restricts to an continuous surjective group homomorphism $\verb!^!:\widehat{G} \rightarrow \widehat{G}_c$.
\end{theorem}

In fact, one has that $\verb!^!$ is also an isomorphism of abstract groups as long as the characters of $G$ separate points on $G$.  This is true for locally compact abelian groups and, as we shall see in the next section, for direct or inverse limits of locally compact abelian groups.
 
\section{Direct Limits of Abelian Groups}
\noindent The famous Pontryagin Duality Theorem asserts that for any locally compact abelian group $G$, there is a canonical topological group isomorphism between $G$ and $\widehat{\widehat{G}}$ given by identifying $g\in G$ with the character $\widehat{g}$ given by $\widehat{g}(\chi) = \chi(g)$ for all $\chi\in \widehat{G}$.  In fact, it is possible to extend this result to the case of direct limits  as we now show.

Suppose that $G=\varinjlim G_n = \cup_{n\in\N} G_n$ is a strict direct limit of locally compact abelian groups.  There are natural continuous projections $p_n:\widehat{G}_{n+1} \rightarrow \widehat{G}_n$ given by $p_n(\chi) = \chi|_{G_n}$ for each character $\chi\in \widehat{G}_n$.  In fact, it is clear that $p_n$ is a homomorphism, and one can show that it is surjective.  We may thus construct a projective-limit group $\varprojlim \widehat{G}_n$.  One then proves the following theorem:

\begin{theorem}(\cite[p. 45]{Sa})
\label{PontryaginDuality}
 The group $G=\varinjlim G_n$ satisfies the Pontryagin duality.  In fact, $  \widehat{\varinjlim G_n} = \varprojlim \widehat{G}_n$ and  $\widehat{\varprojlim  \widehat{G}_n} = \varinjlim G_n$.   
\end{theorem}

In particular, we see that $\widehat{G}$ separates points on the direct-limit group $G=\varinjlim G_n$ and thus that the groups $\widehat{G}$ and $\widehat{G}_c$ are isomorphic as abstract groups. Thus $L^2_M(G)$ decomposes into a direct sum of irreducible representations as follows:
\[
L^2_M(G) \cong \bigoplus_{\chi\in \widehat{G}} \C \chi,
\]
where $\C\chi$ is the one-dimensional subspace of $L^2_M(G)$ generated by the character $\chi\in\widehat{G}$.  Note that $\widehat{G}$ may contain an uncountable number of characters on $G$, in which case $L^2_M(G)$ is an nonseparable Hilbert space.

We end this section with an example.  Consider the Torus group $\T=S^1$.  For each $n$, consider $n^\text{th}$ Cartesian power $\T^n$.  By using the embedding $\T^n\rightarrow \T^{n+1}$ given by  $z\mapsto (z,1)$, one can construct the direct-limit group $\T^\infty = \varinjlim \T^n$.  We recall that $\widehat{T} = \Z$ and that $\widehat{\T^n} = \Z^n$.  Hence, Theorem~\ref{PontryaginDuality} implies that $\widehat{T^\infty} = \varprojlim \Z^n$, where the projections $\Z^{n+1}\rightarrow \Z^n$ are the canonical projections given by $(z,m)\mapsto z$ for all $m\in\Z$ and $z\in \Z^n$. Thus $\widehat{\T^\infty}$ is isomorphic to the the group $\Z^\N$ of all sequences of integers.  Thus,
\[
L^2_M(\T^\infty) \cong \bigoplus_{\sigma\in \Z^\N} \C\chi_\sigma, 
\]
where $\chi_\sigma \in \widehat{\T^\infty}$ is the character corresponding to $\sigma\in\Z^n$.  In particular, $L^2_M(\T^\infty)$ is far from being separable.

Because $\T^\infty$ is abelian, there exist invariant means on $\RUC(\T^\infty)$.  For any such invariant mean $\mu$, one can construct the corresponding regular representation $L^2_\mu(\T^\infty)$.  One has that $L^2_M(\T^\infty)$ is a subrepresentation of $L^2_\mu(\T^\infty)$ for any mean $\mu$.  However, due to the necessity of applying the axiom of choice to construct such an invariant mean, it is not clear whether or not it is possible to say much about the orthogonal complement $L^2_\mu(\T^\infty)\ominus L^2_M(\T^\infty)$ for a mean $\mu$.

\section{Spherical Functions and Direct Limits}
\noindent
 In this section we see how invariant means may be used to describe the behavior of spherical functions on direct limits of Gelfand pairs.  In the classical theory, spherical functions are critically important in studying harmonic analysis on Gelfand pairs.  
 
  We remind the reader that a \textbf{Gelfand pair} is a pair $(G,K)$ of groups, where $G$ is locally compact and $K$ is compact, such that the convolution algebra on the space $L^1(K\backslash G /K)$ of Haar-integrable bi-$K$-invariant functions on $G$ is abelian.  Riemannian symmetric pairs provide the most important examples of Gelfand spaces.  A \textbf{spherical function} on $G$ is function $\phi\in C(G)$ such that 
  \begin{equation}
  \label{sphericalDefinition}
  \int_K \phi(xky)dk = \phi(x)\phi(y)
  \end{equation}
for all $x,y\in G$, where again integration over the compact group $K$ is with the normalized Haar measure.  An irreducible  unitary representation $(\pi,\cH)$ of $G$ is said to be a \textbf{spherical representation} if $\cH^K\neq \{0\}$, where $\cH^K$ is the space of all vectors $v\in \cH$ such that $\pi(k)v=v$ for all $k\in K$.  In fact, for a Gelfand pair $(G,K)$, one can show that $\dim \cH^K= 1$ for every irreducible unitary spherical representation $(\pi,\cH)$ of $G$.  

One can show that the positive-definite spherical functions on $G$ are precisely the matrix coefficients
\[
\phi_\pi(g) = \langle \pi(g)v,v\rangle,
\]
where $\pi$ is an irreducible unitary spherical representation of $G$ and $v$ is a unit vector in $\cH^K\backslash\{0\}$.  This connection between spherical functions and spherical representations allows one to determine the Plancherel decomposition of the quasiregular representation of $G$ on $L^2(G/K)$. Proofs of these classical theorems may be found, for instance, in \cite{vanDijk}.

Suppose one has an increasing family of locally compact groups $\{G_n\}_{n\in\N}$ and an increasing family of compact groups $\{K_n\}_{n\in\N}$ such that $K_n\leq G_n$ for each $n\in\N$ and $(G_n,K_n)$ is a Gelfand pair.  Let $G_\infty = \varinjlim G_n$ and $K_\infty = \varinjlim K_n$.  If we make the additional assumption that $G_n\cap K_{n+1} = K_n$ for all $n\in\N$, then there is a well-defined $G_n$-equivariant inclusion $G_n/K_n \rightarrow G_{n+1}/K_{n+1}$ given by $gK_n\mapsto gK_{n+1}$ for all $g\in G_n$, and we can write $G_\infty/K_\infty = \varinjlim G_n/K_n$.  In this case we say that $(G_\infty,K_\infty)$ is a \textbf{direct-limit spherical pair}.  This definition provides a natural infinite-dimensional generalization of the notion of a Gelfand pair.

It is natural to say that an irreducible unitary representation $(\pi,\cH)$ of $G_\infty$ is a \textbf{spherical representation} if the space $\cH^{K_\infty}$ of $K_\infty$-fixed vectors in $\cH$ is nontrivial.  It is possible to show  (see \cite[Theorem 23.6]{Ol1990}) that, just as for finite-dimensional Gelfand pairs, $\dim \cH^{K_\infty} = 1$ for every irreducible unitary spherical representation $(\pi,\cH)$.

The proper definition of a spherical function is slightly more subtle, because there is no Haar measure on $K_\infty$ over which to integrate.   However, because $K_\infty$ is a direct limit of compact groups, it is amenable, and we may generalize (\ref{sphericalDefinition}) by replacing the Haar measure on $K$ with an invariant mean $\mu$ on $K_\infty$.  That is, we say that a continuous function $\phi\in G_\infty$ is a \textbf{spherical function with respect to $\mu$} if 
\begin{equation}
\label{sphericalDefinitionInfinite}
    \int_{K_\infty} \phi(xky) d\mu(k) = \phi(x)\phi(y)
\end{equation}
for all $x,y\in G$. 

At this point, we remind the reader that, if we define for each $n\in \N$ a mean $\mu_n\in\gM(\RUC(K_\infty))$ by $\mu_n(f) = \int_{K_n} (f|_{K_n})(k) dk$, then any weak-$*$ cluster point of $\{\mu_n\}_{n\in\N}$ is an invariant mean on $K_\infty$.  While this sequence has many cluster points, none of which may be constructed as functionals on all of $\RUC(K_\infty)$ without recourse to the Axiom of Choice, what we can say is that any such $K_\infty$-invariant mean $\mu$ that is a weak-$*$ cluster point of $\{\mu_n\}_{n\in\N}$ must have the property that
\[
\mu(f) = \lim_{n\rightarrow \infty} \int_{K_n} (f|_{K_n})(k) dk
\]
for all functions $f\in\RUC(K_\infty)$ such that the limit on the right-hand side of the equation exists.  We now fix such invariant mean $\mu$ in the closure of $\{\mu_n\}_{n\in\N}$.  For any spherical function $\varphi\in C(G_\infty)$, it follows that if $\varphi$ satisfies
\begin{equation}
\label{limitSpherical}
\lim_{n\rightarrow\infty}\int_{K_n} \varphi(xky) dk = \varphi(x)\varphi(y),
\end{equation}
then $\varphi$ is spherical for $\mu$.

In fact, Olshanski defines a function $f\in C(G_\infty)$ to be spherical if it satisfies (\ref{limitSpherical}).  Note that this condition is stronger than requiring $f$ to be spherical for every invariant mean in the closure of $\{\mu_n\}_{n\in\N}$. However, we will show in Theorem~\ref{sphericalEquivalence} that these two conditions are in fact equivalent.

First we need a lemma about projection operators.  If $G$ is a topological group, $K$ is a compact subgroup, and $(\pi,\cH)$ is any unitary representation of $G$, then the orthogonal projection operator $P:\cH\rightarrow \cH^K$ may be written as
\[
P(v)=\int_K \pi(k)v dk
\]
for all $v\in \cH$.  In fact, using invariant means it is possible to describe the projection operator $P:\cH\rightarrow \cH^{K_\infty}$ for a unitary representation $(\pi,\cH)$ of $G_\infty$ in a completely analogous fashion, as the next lemma shows.  The proof is extremely similar to the proof in the finite-dimensional context, although some care must be taken due to the fact that means do not satisfy the same properties as integrals.

\begin{lemma}
\label{operatorValuedMean}
Suppose that $(\pi,\cH)$ is a unitary representation of a group $G$ and that $K$ is a subgroup of $G$ that is amenable.  Let $\mu$ be an invariant mean in $\gM(\RUC(K_\infty))$.  Then there is a bounded operator $P\in \mathrm{B}(\cH)$ such that
\[
\langle w, Pv \rangle = \int_K \langle w, \pi(k)v\rangle d\mu(k)
\]
for all $v,w\in \cH$.  Furthermore, $P$ is a projection from $\cH$ onto $\cH^K$.  Finally, $P$ is the orthogonal projection if $\mu$ is an inversion-invariant mean.
\end{lemma}
\begin{proof}
 Recall that the matrix coefficient function $k\mapsto \langle w,\pi(k)v\rangle$ is uniformly continuous for all $v,w\in \cH$.  Now fix $v\in \cH$. Because $\pi$ is unitary, we have that $ |\langle w, \pi(k)v\rangle|\leq ||w|| ||v||$ for all $k\in K$ and $w\in \cH$.  Hence
\begin{equation}
\label{boundedness}
\left| \int_K \langle w, \pi(k)v\rangle  d\mu(k) \right| \leq ||w|||v||
\end{equation}
because $\mu$ is a mean.  Thus the map $w\mapsto  \int_K \langle w, \pi(k)v\rangle  d\mu(k)$ defines a bounded linear functional on $\cH$ and by the Riesz representation theorem there is a unique vector $Pv\in \cH$ such that 
\[
\langle w, Pv \rangle = \int_K \langle w, \pi(k)v\rangle d\mu(k)
\]
for all $w\in \cH$.  It is clear that $v\mapsto Pv$ is linear.  Furthermore, $P\in\mathrm{B}(\cH)$ by (\ref{boundedness}).

Next we see that $Pv\in \cH^K$ for all $v\in\cH$.  In fact, for all $h\in K$, we have that 
\begin{align*}
\langle w, \pi(h)Pv\rangle &= \langle \pi(h^{-1})w,Pv \rangle \\
            &= \int_K \langle \pi(h^{-1})w, \pi(k)v \rangle d\mu(k) \\
            &= \int_K \langle w, \pi(hk)v \rangle d\mu(k)\\
            &= \int_K \langle w, \pi(k)v \rangle d\mu(k)
            = \langle w, Pv\rangle,
\end{align*}
and thus $\pi(h)Pv= Pv$ for all $h\in K$.  Similarly, $Pv = v$ for all $v\in \cH^K$.  In fact,
\begin{align*}
\langle w, Pv\rangle   &= \int_K \langle w, \pi(k)v \rangle d\mu(k) \\
            &= \int_K \langle w, v \rangle d\mu(k) = \langle w,v\rangle 
\end{align*}
for all $w\in \cH$ and $v\in \cH^K$.  Thus $P$ is a projection onto $\cH^K$.

It only remains to be shown that $P$ is self-adjoint if $\mu$ is inversion-invariant. For any $v,w\in \cH$, we have
\begin{align*}
\langle w, Pv \rangle &= \int_K \langle w, \pi(k)v\rangle d\mu(k) \\
    &= \int_K \langle \pi(k^{-1})w, v\rangle d\mu(k) \\
    &= \int_K \langle \pi(k)w, v\rangle d\mu(k) \\
      &= \int_K \langle \overline{ v, \pi(k)w}\rangle d\mu(k) \\
      &=\overline{ \int_K \langle  v, \pi(k)w\rangle d\mu(k)} = \overline{\langle v, Pw\rangle},    
\end{align*}
where we have used the fact that $\mu$ is inversion-invariant and that $\mu(\overline{f}) = \overline{\mu(f)}$ for all $f\in \RUC(K)$.
\end{proof}

We are now ready to show that the condition (\ref{sphericalDefinitionInfinite}) is independent of the choice of mean.  Again, the proof is almost entirely analogous to the proof in the finite-dimensional context.  We remark that Olshanski showed in \cite[Theorem 23.6]{Ol1990} that conditions (3) and (4) in the following theorem are equivalent using a different method.

\begin{theorem}
\label{sphericalEquivalence}
Suppose that $(G_\infty,K_\infty)$ is a direct-limit Gelfand pair, and let $\varphi:G\rightarrow\C$ be a positive-definite function such that $\varphi(e) =1$.  Then the following are equivalent:
\begin{enumerate}
\item $\varphi$ is spherical for every invariant mean $\mu\in\gM(\RUC(K_\infty))$.
\item There exists an invariant mean $\mu\in\gM(\RUC(K_\infty))$ with respect to which $\varphi$ is spherical.
\item There exists an irreducible unitary spherical representation $(\pi,\cH)$ of $G_\infty$ such that
\[
   \varphi(g) = \langle \pi(g)v, v\rangle,
\]
where $v\in \cH^{K_\infty}$ is a unit vector.

\item $\varphi$ satisfies 
\[ \lim_{n\rightarrow \infty} \int_{K_n} \phi(xky) dk = \phi(x)\phi(y).\]
\end{enumerate}
\end{theorem}
\begin{proof}
Because $\varphi$ is a positive-definite function with $\varphi(e) = 1$, we can use the Gelfand-Naimark-Segal construction to construct a unitary representation $(\pi,\cH)$ of $G_\infty$ and a cyclic unit vector $v\in \cH$ such that $\varphi(g) = \langle \pi(g)v,v\rangle$.

That (1)$\implies$(2) is clear.  To prove that (2)$\implies$(3), suppose that $\varphi$ is spherical for an invariant mean $\mu$.  Then we claim that $\varphi$ is right-$K_\infty$-invariant.  In fact, we have that 
\begin{align*}
\varphi(xh)
            &= \int_{K_\infty} \varphi(xhk)d\mu(k) \\
            &=  \int_{K_\infty} \varphi(xh)d\mu(k) 
            = \varphi(x)
\end{align*}
for any $h\in K_\infty$, where we use that $\varphi(e)=1$.
The proof that $\varphi$ is right-$K_\infty$-invariant is identical.  It follows that 
\begin{align*}
\langle \pi(k)v,\pi(g)v\rangle &= \langle \pi(g^{-1}k)v,  v\rangle \\
                               &= \langle \pi(g^{-1})v,  v\rangle = \langle v, \pi(g)v\rangle
\end{align*}
for all $g\in G$ and $k\in K$.  Because $v$ is a cyclic vector in $\cH$, we see that $\pi(k)v=v$ for all $k\in K$. 

It remains to be shown that $\dim \cH^{K_\infty} = 1$ and thus that $\pi$ is irreducible. But
\begin{align*}
   \left\langle  P(\pi(y)v),\pi( x^{-1})v\right\rangle  
        & = \int_K \langle  \pi(k)\pi(y)v, \pi(x^{-1})v \rangle d\mu(k) \\
        & = \int_K \phi(xky) d\mu(k)  \\
        & = \varphi(x)\varphi(y) \\
        & = \langle  \varphi(y)v, \pi(x^{-1})v \rangle
\end{align*}
for all $x,y\in G$, where $P:\cH\rightarrow\cH^{K_\infty}$ is the projection operator defined in Lemma~\ref{operatorValuedMean}. Because $v$ is cyclic, it follows that $P(\pi(y)v) = \varphi(y)v$ for all $y\in G$.  Using again the fact that $v$ is cyclic, we see that $\dim (\text{range } P) = 1$.  In other words, $\dim \cH^{K_\infty} = 1$, and thus $\cH$ is irreducible.

Finally we prove (3)$\implies$(1).  Suppose that $\pi$ is an irreducible spherical representation with $v\in\cH^{K_\infty}$ and that $\mu$ is an invariant mean in $\gM(\RUC(K_\infty))$. We need to show that $\langle \varphi(g)v, v\rangle$ is spherical with respect to $\mu$.  

As before, we consider the projection $P:\cH\rightarrow \cH^{K_\infty}$ defined in Lemma~\ref{operatorValuedMean}. Since $P(\pi(y)v)\in \cH^K$ and $\mathrm{dim} \cH^K = 1$, it follows that $P(\pi(y)v) = cv$ for some nonzero $c\in\C$.  But then
\begin{align*}
   c & = \langle P(\pi(y)v), v \rangle \\
     & = \int_{K_\infty} \langle \pi(ky)v, v\rangle d\mu(k)\\
     & =  \langle \pi(y)v,\pi(k^{-1})v\rangle = \langle \pi(y)v,v\rangle,
\end{align*}
since $v$ is $K_\infty$-invariant.
Hence
\begin{align*}
\int_{K_\infty} \varphi(xky) dk & = \int_{K_\infty} \langle  \pi(xky) v,v\rangle d\mu(k)\\
        & = \left\langle \int \pi(k) \pi(y)v, \pi(x^{-1})v\ d\mu(k) \right\rangle\\
        & = \left\langle P(\pi(y)v), \pi(x^{-1})v \right\rangle\\
        & = \left\langle \langle \pi(y)v,v\rangle v, \pi(x^{-1})v \right\rangle\\
        & = \left\langle  \pi(x)v, v\right\rangle \left\langle  \pi(y)v, v\right\rangle\\
        & = \varphi(x)\varphi(y).
\end{align*}
Thus $\phi$ is spherical for $\mu$.

We have already seen that (4)$\implies$(2) (see the discussion surrounding (\ref{limitSpherical})).
 Finally, we demonstrate that (3)$\implies$(4).  Suppose that $\varphi(g) = \langle \pi(g)v,v\rangle$, where $\pi$ is an irreducible spherical representation with $v\in\cH^{K_\infty}$. We know from the preceding paragraph that $\phi$ is a spherical function with respect to every invariant mean on $\RUC(K_\infty)$. 
 
 For each $n\in\N$, we consider the orthogonal projection operator $P_n:\cH\rightarrow \cH^{K_n}$, which may be written as
 \[
 P_n(v) = \int_{K_n} \pi(k)v dk
 \]
 for all $v\in\cH$.
Note also that $\cH^{K_\infty} = \cap_{n\in\N} \cH^{K_n}$.  Consider the orthogonal projection $P:\cH\rightarrow \cH^{K_\infty}$.  Then $P_n\rightarrow P$ in the strong operator topology on $\cH$.  In other words, we have that
\[
P(v) = \lim_{n\rightarrow\infty} \int_{K_n} \pi(k)v dk
\]
for all $v\in\cH$.  
Now let $\mu$ be a $K_\infty$-invariant mean on $\RUC(K_\infty)$ which is also inversion invariant.  Then from Lemma~\ref{operatorValuedMean} we have that the orthogonal projection $P:\cH\rightarrow \cH^{K_\infty}$ satisfies 
\[
\langle Pv,w\rangle = \int_{K_\infty} \langle \pi(k)v,w\rangle d\mu(k).
\]
Hence, because $\phi$ is spherical for $\mu$, we see that
\begin{align*}
\phi(x)\phi(y) &= \int_{K_\infty}\phi(xky)d\mu(k) \\
                & = \left\langle P(\pi(y)v), \pi(x^{-1})v \right\rangle\\
                & = \lim_{n\rightarrow\infty}  \int_{K_n} \langle \pi(ky)v,\pi(x^{-1})v \rangle \\
                & = \lim_{n\rightarrow\infty} \int_{K_n} \varphi(xky) d\mu(k),
\end{align*}
and we are done.
\end{proof}

 It should be mentioned that for many direct-limit Gelfand  pairs, there is a rich collection of spherical functions which have already been classified (see, for instance, \cite{DOW,Ol1990}). However, some peculiar behaviors arise in this infinite-dimensional context that do not occur in the finite-dimensional theory.   Olshanski (\cite[Corollary 23.9]{Ol1990}) has shown that for the classical direct limits of symmetric spaces (that is, those formed by direct limits of classical matrix groups with embeddings of the form $A\mapsto\left( \begin{array}{ll} A & 0 \\ 0 & 1\end{array}\right)$), the product of two spherical functions is again spherical. See also \cite{Voicu1, Voicu2} for a different proof in a special case.  We recall that the classical direct-limit groups $\SO(\infty)$, $\SU(\infty)$, and their direct products are extremely amenable.   In fact, the following corollary of Theorem~\ref{sphericalEquivalence} shows how this surprising multiplicative property of spherical functions on $G_\infty$ is related to extreme amenability of $K_\infty$.  

\begin{corollary}
If $(G_\infty,K_\infty)$ is a direct-limit Gelfand pair such that $K_\infty$ is extremely amenable, then the product of two spherical functions is again a spherical function.
\end{corollary}
\begin{proof}
Suppose that $\varphi$ and $\psi$ are spherical functions on $G_\infty$.  Because $K_\infty$ is extremely amenable, there is a $K_\infty$-invariant character $\mu$ on the $C^*$-algebra $\RUC(K_\infty)$.  Then
\begin{align*}
\int_{K_\infty} (\varphi \psi)(xky) d\mu(k) &= \int_{K_\infty} \varphi(xky) d\mu(k) \int_{K_\infty} \psi(xky) d\mu(k) \\
 & = \varphi(x)\varphi(y)\psi(x)\psi(y) \\
 & = (\varphi\psi)(x)(\varphi\psi)(y)
\end{align*}
Thus $\varphi\psi$ is spherical.
\end{proof}


We end by remarking that the natural way in which invariant means may be used as a replacement for integration in the context of spherical functions suggests to the authors that there may be other opportunities to apply amenability theory to the study of representations and harmonic analysis on direct-limit groups.

\end{document}